\documentclass[brochure,english,12pt]{bourbaki}
\usepackage[matrix,arrow]{xy}
\usepackage{amssymb,amsfonts,amsmath,footnote}
\usepackage[francais]{babel}
\usepackage{url}
\date{Avril 2011}
\bbkannee{63\`eme ann\'ee, 2010-2011}
\bbknumero{1035}
\title{The fundamental lemma and the Hitchin fibration}
\subtitle{after Ng\^o Bao Ch\^au}
\author{Thomas C. HALES}
\address{University of Pittsburgh\\
Department of Mathematics\\
Pittsburgh, PA 15260 -- U.S.A.}
\email{hales@pitt.edu}

\theoremstyle{plain}
\newtheorem{example}[equation]{Example}
\newtheorem{definition}[equation]{Definition}
\newtheorem{theorem}[equation]{Theorem}

\newtheorem{corollary}[equation]{Corollary}

\def\op#1{{\operatorname{#1}}}
\newcommand{\ring}[1]{\mathbb{#1}}

\def\card{\op{card}}
\def\a{{\scriptsize\text{ani}}}
\def\good{{\scriptsize\text{good}}}

\def\SO{{\mathbf {SO}}}
\def\OO{{\mathbf O}}

\def\so{\mathfrak{so}}
\def\sp{\mathfrak{sp}}
\def\gl{\mathfrak{gl}}
\def\sl{\mathfrak{sl}}
\def\g{\mathfrak{g}}
\def\t{\mathfrak{t}}
\def\cc{\mathfrak{c}}
\def\DIV{{\mathfrak{D}}}
\def\RDIV{{\mathfrak{R}}}

\def\A{{\mathcal A}}
\def\C{{\mathcal C}}
\def\M{{\mathcal M}}
\def\P{{\mathcal P}}
\def\O{{\mathcal O}}
\def\tA{{\tilde{\mathcal A}}}
\def\tP{{\tilde{\mathcal P}}}
\def\tM{{\tilde{\mathcal M}}}
\def\tU{{\tilde{\mathcal U}}}

\begin{document}

\maketitle

{

\narrower{\it The study of orbital integrals on $p$-adic groups has turned
out to be singularly difficult. -- R. P. Langlands, 1992} % Real Igusa

}

\bigskip

This report describes some remarkable identities of integrals that have
been established by Ng\^o Bao Ch\^au.   My task will be to describe
why these identities -- collectively called the fundamental lemma (FL) --
took nearly thirty years to prove, and why they have particular
importance for investigations in the theory of automorphic
representations.

\section{basic concepts}

\subsection{origins of the fundamental lemma (FL)}\label{sec:origin}

To orient ourselves, we give special examples of
behavior that the theory is designed to explain.

\begin{example}\label{ex:sl2}  We recall the definition of the holomorphic discrete series representations
  of $SL_2(\ring{R})$.  For each natural number $n\ge 2$, let
  $V_{n,+}$ be the vector space of all holomorphic functions $f$ on
  the upper half plane ${\mathfrak h}$ such that
\[
\int_{\mathfrak h} |f|^2 y^{n-2} dx\, dy < \infty.
\]
$SL_2(\ring{R})$ acts on $V_{n,+}$:
\[
\begin{pmatrix} a & b \\ c & d \end{pmatrix} \cdot f(z) = 
(-b z + d ) ^{-n} f (\frac{\phantom{-}a z - c}{-b z + d}).
\]
Similarly, for each $n\ge 2$,  there is an anti-holomorphic discrete series
representation  $V_{n,-}$.  These infinite
dimensional representations have characters that exist as locally integrable
functions $\Theta_{n,\pm}$.  The characters are equal:
$\Theta_{n,+}(g)=\Theta_{n,-}(g)$, except when $g$ is conjugate to a
rotation
\[
\gamma = 
\begin{pmatrix} \phantom{-}\cos\theta & \sin\theta \\ -\sin\theta & \cos\theta\end{pmatrix}.
\] 
When $g$ is conjugate to $\gamma$, a remarkable character identity holds:
\begin{equation}\label{eqn:discrete-series}
\Theta_{n,-}(\gamma) - \Theta_{n,+}(\gamma) = 
\frac{e^{i (n-1) \theta} + e^{- i (n-1) \theta}}{e^{i\theta}-e^{-i\theta}}.
\end{equation}
It is striking that numerator of the difference of two characters of
infinite dimensional representations collapses to the character of a two
dimensional representation $\gamma\mapsto \gamma^{n-1}$ of the group $H$ of
rotations.  Shelstad gives general characters identities of this
sort~\cite{Shelstad:OI}. % Shelstad Th. 5.4.1.
% Knapp, page 345, page 35.
\end{example}

We find another early glimpse of the theory in a letter to Singer from
Langlands in 1974~\cite{L:singer:1974}.  Singer had expressed interest
in a particular alternating sum of dimensions of spaces of cusp forms
of $G=SL_2$ over a totally real number field $F$.  Langlands's reply
to Singer describes then unpublished joint work with
Labesse~\cite{LL:1979}.  Without going into details, we remark that in
the calculation of  this alternating sum, there is again a collapse in
complexity from the three dimensional group $SL_2$ to a sum indexed by
one-dimensional groups $H$ (of norm $1$ elements of totally imaginary
quadratic extensions of $F$).

These two examples fit into a general framework that have now led to
major results in the theory of automorphic representations and number
theory, as described in Section~\ref{sec:uses}.  Langlands holds that
methods should be developed that are adequate for the theory of
automorphic representations in its full natural generality.  This
means going from $SL_2$ (or even a torus) to all reductive groups,
from one local field to all local fields, from local fields to global
fields and back again, from the geometric side of the trace formula to
the spectral side and back again.  Moreover, interconnections between
different reductive groups and Galois groups should be included, as
predicted by his general principle of functoriality.

Thus, from these early calculations of Labesse and Langlands, the
general idea developed that one should account for alternating sums
(or $\kappa$-sums as we shall call them because they occasionally 
involve roots of unity other than $\pm1$) that appear in the harmonic
analysis on a reductive group $G$ in terms of the harmonic analysis on
groups $H$ of smaller dimension.  The FL is a concrete
expression of this idea.

\subsection{orbital integrals}

This section provides brief motivation about why researchers care
about integrals over conjugacy classes in a reductive group.  Further
motivation is provided in Section~\ref{sec:uses}.

It is a basic fact about the representation theory of a finite group
that the set of irreducible characters is a basis of the vector space
of class functions on the group.  A second basis of that vector space
is given by the set of characteristic functions of the conjugacy
classes in the group.  We will loosely speak of any linear relation
among the set of characteristic functions of conjugacy classes and the
set of irreducible characters as a {\it trace formula}.

More generally, we consider a reductive group $G$ over a local field.
Each admissible representation $\pi$ of $G$ defines a {\it
  distribution character:}
\[
f\mapsto
\op{trace}\,\int_G f(g)\pi(g)\,dg,\quad  f\in C_c^\infty(G),
\]
with $dg$  a Haar measure on $G$.
A trace formula in this context should be a linear
relation among characteristic functions of conjugacy classes and
distribution characters.  To put all terms of a trace formula on equal
footing, the characteristic function of a conjugacy class must also be
treated as a distribution, called an {\it orbital integral}:
\[
f\mapsto \OO(\gamma,f) = \int_{I_\gamma\backslash G} f(g^{-1}\gamma g)\, dg,\quad f\in C_c^\infty(G),
\]
where $I_\gamma$ is the centralizer of $\gamma\in G$.

The FL is a collection of identities among orbital integrals that may be used in a trace formula to
obtain identities among representations $\pi$.

\subsection{stable conjugacy}\label{sec:stable}

At the root of these $\kappa$-sum formulas is the distinction between
{\it ordinary conjugacy} and {\it stable conjugacy}.  
\begin{example}
A clockwise rotation and counterclockwise rotation
\[
\begin{pmatrix}\cos\theta &-\sin\theta\\\sin\theta &\phantom{-}\cos\theta\end{pmatrix}
\text{~~and~~}
\begin{pmatrix}\phantom{-}\cos\theta &\sin\theta\\-\sin\theta &\cos\theta\end{pmatrix}
\]
in $SL_2(\ring{R})$ are conjugate by the complex matrix $\begin{pmatrix}i&0\\0&-i\end{pmatrix}$, 
but they are not conjugate
in the group $SL_2(\ring{R})$ when $\theta\not\in\ring{Z}\pi$.  Indeed,
a matrix calculation shows that every element of $GL_2(\ring{R})$ that
conjugates the rotation to counter-rotation has odd determinant,
thereby falling outside $SL_2(\ring{R})$.  Alternatively, they are not
conjugate in $SL_2(\ring{R})$ because the character identity
(\ref{eqn:discrete-series}) separates them.
\end{example}

Let $G$ be a reductive group defined over a field
  $F$ with algebraic closure $\bar F$. 

\begin{definition}  An element
  $\gamma'\in G(F)$ is said to be {\it stably conjugate} to a given regular
  semisimple element $\gamma\in G(F)$ if $\gamma'$ is conjugate to $\gamma$ in
  the group $G(\bar F)$.
\end{definition}

There is a Galois cohomology group that can be used to study the conjugacy classes within
a given stable conjugacy class.  Let $I_\gamma$ be the centralizer of an element $\gamma\in G(F)$.
The centralizer is a Cartan subgroup when $\gamma$ is a (strongly) regular semisimple element.
Write $\gamma'=g^{-1}\gamma g$, for $g\in G(\bar F)$.
For every element $\sigma$ of the Galois group $\op{Gal}(\bar F/F)$,  we have
 $g\,\sigma(g)^{-1}\in I_\gamma(\bar F)$.  These elements define  in
the Galois cohomology group $H^1(F,I_\gamma)$ a class, which does not depend
on the choice of $g$.  It is the trivial class when
$\gamma'$ is conjugate to $\gamma$.

\begin{example} The centralizer $I_\gamma$ of a regular rotation
  $\gamma$ is the subgroup of all rotations in $SL_2(\ring{R})$.  The
  group $I_\gamma(\ring{C})$ is isomorphic to $\ring{C}^\times$.  Each
  cocycle is determined by the value $r\in
  I_\gamma(\ring{C})=\ring{C}^\times$ of the cocycle on the generator of
  $\op{Gal}(\ring{C}/\ring{R})$.  A given $r\in\ring{C}^\times$
  satisfies the cocycle condition when $r\in \ring{R}^\times$ and
  represents the trivial class in cohomology when $r$ is positive.
  This identifies the cohomology group:
\[
H^1(\ring{R},I_\gamma) = \ring{R}^\times/\ring{R}^\times_+ = \ring{Z}/2\ring{Z}.
\]
This cyclic group of order two classifies the two conjugacy classes
within the stable conjugacy class of a rotation.
\end{example}

When $F$ is a local field, $A=H^1(F,I_\gamma)$ is a
finite abelian group.  Every function $A\to\ring{C}$  has a
Fourier expansion as a linear combination of characters $\kappa$ of
$A$.  
The {\it theory of endoscopy} is the subject that studies stable
conjugacy through the separate characters $\kappa$ of $A$.  Allowing
ourselves to be deliberately vague for a moment, the idea of endoscopy
is that the Fourier mode of $\kappa$ (for given $I_\gamma$ and $G$)
produces oscillations that cause some of the roots of $G$ to cancel
away.  The remaining roots are reinforced by the oscillations and
become more pronounced.  The root system consisting of the pronounced
roots defines a group $H$ of smaller dimension than $G$.  With respect
to the harmonic analysis on the two groups, the mode of $\kappa$ on
the group $G$ should be related to the dominant mode on $H$.

\subsection{endoscopy}\label{sec:endoscopy}

The smaller group $H$, formed from the ``pronounced'' subset of the
roots of $G$, is called an endoscopic group.  Hints about how to
define $H$ precisely come from various sources.
\begin{itemize}
\item It should be constructed from the data $(G,I_\gamma,\kappa)$, with $\gamma$ regular semisimple.
\item Its roots should be a subset of the roots of $G$ (although $H$ need not
be a subgroup of $G$).
\item $H$ should have a Cartan subgroup $I_H\subset H$ isomorphic over $F$ to the Cartan subgroup
$I_\gamma$ of $G$, compatible with the Weyl groups of the two groups $H$ and $G$.
\item Over a nonarchimedean local field, the spherical Hecke algebra on $G$ should be related to the
spherical algebra on $H$.
\item It should generalize the example of Labesse and Langlands.
\end{itemize}

Every reductive group $G$ has a {\it dual group} $\hat G$ that is defined over $\ring{C}$.
The character group of a Cartan subgroup in the
dual group is the cocharacter group of a Cartan subgroup in $G$, and
the roots of the $\hat G$ are the coroots of $G$.  The dual of a
semisimple simply connected semisimple group is an adjoint group, and
vice versa.  For example, we have dualities $\hat GL(n) = PGL(n)$ and
$\hat Sp(2n)=SO(2n+1)$.  The duality between the root systems of
$Sp(2n)$ and $SO(2n+1)$ interchanges short and long roots.  The groups
$G$ and $\hat G$ have isomorphic Weyl groups.  We write $\hat T\subset
\hat G$ for a Cartan subgroup of $\hat G$.  There is a somewhat larger
dual group ${}^LG$ that is defined as a semidirect product of $\hat G$
with the Galois group of the splitting field of $G$.

There are indications that the groups $H$ should be defined through  
the dual $\hat G$ (or more precisely, ${}^LG$) of $G$:
\begin{itemize}
\item Langlands's principle of functoriality is a collection of
  conjectures, relating  the representation theory of groups
  when their dual groups are related.  Since the
  examples about $SL_2$ in Section~\ref{sec:origin} are representation
  theoretic, we should look to the dual.
\item The Satake transform identifies the spherical Hecke algebra with
  a dual object.  
\item The Kottwitz-Tate-Nakayama isomorphism identifies the group of
  characters on $H^1(F,I_\gamma)$ with a subquotient $\pi_0(\hat T^\Gamma)$
  of the dual torus $\hat T$.  (This subquotient is the group of
  components of the set of fixed points of $\hat T$ under an action of the Galois group of the
splitting field of $I_\gamma$.)
\end{itemize}

\begin{definition}[endoscopic group]  Let $F$ be a local field.
  The {\it endoscopic group} $H$ associated with $(G,I_\gamma,\kappa)$ is defined
  as follows.  By the Kottwitz-Tate-Nakayama isomorphism just
  mentioned, $\kappa$ is represented by an element of the dual torus,
  $\hat T$.  By an abuse of notation, we will also write $\kappa\in
  \hat T$ for this element.  The identity component of the
  centralizer of $\kappa$ is the dual $\hat H$ of a quasi-split
  reductive group $H$ over $F$.  The choice of a particular quasi-split form $H$
  among its outer forms is determined by the condition that there should
  be an isomorphism over $F$ of a Cartan subgroup $I_H$ of $H$ with
  $I_\gamma$ in $G$, compatible with their respective Weyl group actions.
\end{definition}

We write $\rho$ for the choice of quasi-split form $H$ among its outer
forms and refer to the pair $(\kappa,\rho)$ as {\it endoscopic data}
for $H$.  More generally, if $G$ is defined over any field, we can use
a pair $(\kappa,\rho)$, with $\kappa\in \hat T$, to define
an endoscopic group $H$ over that same field.

One of the challenging aspects of the FL is that it is an assertion of
direct relation between groups that are defined by a dual
relation.  Very limited information (such as
Cartan subgroups, root systems, and Weyl groups) can be transmitted
from the endoscopic group $H$ to $G$ through the dual group.

\section{a bit of Lie theory}

\subsection{characteristic polynomials}\label{sec:chevalley}

Let $G$ be a split reductive group over a field $k$ and let
$\g$ be its Lie algebra, with split Cartan subalgebra $\t$ and Weyl
group $W$.  We assume throughout this report that the characteristic of $k$ is sufficiently
large (more than twice the Coxeter number of $G$, to be precise).  The group
$G$ acts on $\g$ by the adjoint action.  By Chevalley,  the restriction  of regular
functions from $\g$ to $\t$ induces an isomorphism
\[
k[\g]^G = k[\t]^W.
\]
We let $\cc =  \op{Spec}\,(k[\t]^W)$, and let $\chi:\g\to\cc$ be the morphism deduced from
Chevalley's isomorphism.  The following example shows that $\chi:\g\to\cc$ is a generalization
of the characteristic polynomial of a matrix.

\begin{example}
If $G=GL(n)$, then $k[\g]^G$ is a polynomial ring,  generated by the coefficients $c_i$ of the characteristic
polynomial 
\begin{equation}
p(t)=t^n + c_{n-1} t^{n-1} +\cdots +c_0
\end{equation}
of a matrix $a\in \g=\gl(n)$.  The morphism $\chi:\g\to\cc$ can be identified with
the ``characteristic map'' that sends $a$ to $(c_{n-1},\ldots,c_0)$.
\end{example}

\subsection{Kostant section}

Kostant constructs a section $\epsilon:\cc\to\g$ of $\chi:\g\to\cc$
whose image lies in the set $\g^{reg}$ of regular elements of $\g$.
In simplified terms, this constructs a matrix with a given
characteristic polynomial.

\begin{example}
  When $\g=\sl(2)$, the Lie algebra consists of matrices of trace zero, and the characteristic polynomial 
   has the form $t^2 +c$.  The determinant $c$ generates  $k[\g]^G$.  
  The Kostant section maps $c$ to
\begin{equation}\label{ex:kostant}
\begin{pmatrix} 0 & -c\\ 1 & 0\end{pmatrix}.
\end{equation}
\end{example}

\begin{example}
  If $\g=\gl(n)$, we can construct the companion matrix of a given
  characteristic polynomial $p\in k[t]$, by taking the endomorphism $t$ of
  $R=k[t]/(p)$, expressed as a matrix with respect to the standard
  basis $1,t,t^2,\ldots,t^{n-1}$ of $R$.  The companion matrix is a
  section $\cc\to\g$ that is somewhat different from the Kostant
  section.  Nevertheless, the Kostant section can be viewed as a
  generalization of this that works uniformly for all Lie algebras
  $\g$.
\end{example}

\subsection{centralizers}

Each element $\gamma\in\g$ has a centralizer $I_\gamma$ in $G$.
If two elements of $\g^{reg}$ have the same image $a$ in $\cc$, then
their centralizers are canonically isomorphic.  By descent, there is a
regular centralizer $J_a$, for all $a$ in $\cc$,  that is canonically isomorphic to
$I_\gamma$ for every regular element $\gamma$ such that
$\chi(\gamma)=a$.

\begin{example} Suppose $G=SL(2)$.
  We may identify $J_a$ with the centralizer of~(\ref{ex:kostant}) to obtain the
   group of matrices with determinant $1$ of the form
\[
\begin{pmatrix} x & -y\, c\\ y & \phantom{y}x
\end{pmatrix}.
\]
\end{example}

\begin{example}
  If $\g=\gl(n)$, then the centralizer of the companion matrix with
  characteristic polynomial $p$ can be identified with the centralizer
  of $t$ in $GL(R)$, where $R=k[t]/(p)$.  An
  element of $\gl(R)$ centralizes the regular element $t$ if and only if it is a polynomial
  in $t$.  Thus, the centralizer in $\gl(R)$ is $R$ and the centralizer in $GL(R)$
 is $J_a = R^\times$.
\end{example}

\subsection{discriminant and resultant}

Let $\Phi$ be the root system of a split group $G$.  The differentials $d\alpha$ of 
roots define a  polynomial called the {\it discriminant}:
\begin{equation}\label{eqn:disc}
\prod_{\alpha\in\Phi} d\alpha
\end{equation}
on $\t$.  The polynomial is invariant under the action of the Weyl group 
$W$ and equals a function on $\cc$.
The divisor $\DIV_G$ of this polynomial $\cc$ is called the {\it discriminant divisor}.

\begin{example} Let $G=GL(n)$.  The Lie algebra $\t$ 
  can be identified with the diagonal matrix algebra with coordinates
  $t_1,\ldots,t_n$ along the diagonal.  The discriminant is
\[\prod_{i\ne j} (t_i - t_j).\]
This  is invariant under the action of the symmetric group on $n$ letters 
and can be expressed as a polynomial in the coefficients $c_i$ of the characteristic polynomial.
In particular, the discriminant of the characteristic polynomial $t^2 + b t + c$ is the usual
discriminant $b^2 - 4 c$.
\end{example}

If $H$ is a split endoscopic group of $G$, there is a morphism 
\begin{equation}\label{eqn:nu}
\nu:\cc_H\to \cc
\end{equation}
that comes from an isomorphism of Cartan subalgebras $\t_H\to\t$ and an inclusion
of Weyl groups $W_H\subset W$: $\cc_H = \t/W_H \to \t/W= \cc$.
There exists
a {\it resultant divisor} $\RDIV$ such that
\begin{equation}
\nu^*\DIV_G = \DIV_H + 2\,\, \RDIV.
\end{equation}

\begin{example} Let $H=GL(2)\times GL(2)$, embedded as a block
  diagonal subgroup of $GL(4)$.  Identify roots of $H$ with roots of $G$ under this embedding.
 The morphism $\nu:\cc_H\to \cc$,
  viewed in terms of characteristic polynomials, maps the pair
  $(p_1,p_2)$ of quadratic polynomials to the quartic $p_1p_2$.  Let
  $t^1_i,t^2_i$ be the roots of $p_i$, for $i=1,2$.  The resultant is
\[
\prod_{j\ne k} (t^j_1 - t^k_2).
\]
  The
resultant is symmetric in the roots of $p_1$ and in the roots of $p_2$
and thus can be expressed as a polynomial  in the coefficients
of $p_1$ and $p_2$.  It vanishes exactly when $p_1$ and $p_2$ have a common root.
\end{example}

\section{the statement of the FL}\label{sec:statement}

Let $G$ be a reductive group scheme over the ring of integers $\O_v$ of
a nonarchimedean local field $F_v$ in positive characteristic.  Let $q$ be the cardinality
of the residue field $k$.  The map $\chi:\g\to\cc$
is compatible with stable conjugacy in the sense that two
regular semisimple elements in $\g(F_v)$ are stably conjugate exactly
when they have the same image in $\cc(F_v)$.    The results of
Section~\ref{sec:stable} (adapted to the Lie algebra)
show that each element $\gamma$ stably conjugate to
$\epsilon(a)$ carries a cohomological invariant in $H^1(F_v,J_a)$, which is trivial for
elements conjugate to $\epsilon(a)$. 

For each regular semisimple element $a\in \cc(F_v)$ and character 
\[
\kappa:H^1(F_v,J_a)\to\ring{C}^\times,
\]
we write
 $\langle\kappa,\gamma\rangle$ for the pairing of
of $\kappa$ with the cohomological invariant of $\gamma$.
A 
$\kappa$-orbital integral is defined to be
\begin{equation}\label{eqn:k-orbital}
\OO_\kappa(a) = \sum_ {\chi\gamma= a} 
\int_{I_\gamma\backslash G(F_v)} \langle\kappa,\gamma\rangle 
1_{\g(\O_v)} (\op{Ad}\, g^{-1}(\gamma)) dg,
\end{equation}
where $I_\gamma$ centralizes $\gamma$, and the sum runs over
representatives of the conjugacy classes in $\g(F_v)$ with image $a$.
Here $1_{\g(\O_v)}$ is the characteristic function of $\g(\O_v)$.  A
Haar measure on $G$ has been fixed that gives $G(\O_v)$ volume $1$.

The character $\kappa$ determines a reductive group scheme $H$ over
$\O_v$, according to the construction of~(\ref{sec:endoscopy}).
In general, we add
a subscript $H$ to indicate quantities constructed for $H$, analogous
to those already constructed for $G$.  In particular,
let $\cc_H$ be the Chevalley quotient of the Lie algebra of $H$. There is
a morphism $\nu:\cc_H\to\cc$.    When $\kappa$ is trivial, we
write $\SO$ for $\OO_\kappa$.

Here is the main theorem of Ng\^o~\cite{NBC:2010}.

\begin{theorem}[fundamental lemma (FL)]
  Assume that the characteristic of $F_v$ is greater than twice the
  Coxeter number of $G$.  For all regular semisimple elements $a\in
  \cc_H(\O_v)$ whose image $\nu(a)$ in $\cc$ is also regular
  semisimple, the $\kappa$ orbital integral of $\nu(a)$ in $G$ is
  equal to the stable orbital integral of $a$ in $H$, up to a power of
  $q$:
\[
\OO_\kappa(\nu(a)) = q^{r_v(a)}\SO_H(a),\quad{ where~~ } r_v(a) = \deg_v(a^*\RDIV).
\]
\end{theorem}

A sketch of Ng\^o's proof of the FL  appears in Section~\ref{sec:lmf}.

Over the years from the time that Langlands first conjectured the FL
until the time that Ng\^o gave its proof, the FL has been transformed
into simpler form~\cite{Langlands:debuts}.  The statement of the FL
appears here in its simple form.  Section~\ref{sec:reduce} makes a
series of comments about the original form of the FL and its
reduction to this simple form.  Except for that section, our
discussion is based on this simple form of the FL.  In particular, we
assume that the field $F_v$ has positive characteristic and that the
conjugacy classes live in the Lie algebra rather than the group.

Analogous identities (transfer of Schwartz functions)
on real reductive groups have been established by 
Shelstad~\cite{Shelstad:OI}.  Her work gives a precise form to the
 idea that the oscillations of a character
$\kappa$ cause certain roots to cancel away and others to become more
pronounced: normalized $\kappa$-orbital integrals extend smoothly
across the singular hyperplanes of some purely imaginary roots $\alpha$,
but jump across others.  At a philosophical distance, Ng\^o's use of
perverse sheaves can be viewed as $p$-adic substitute for 
differential operators, introduced by Harish-Chandra to study
invariant distributions near a singular element in the group and
adopted by Shelstad as a primary tool.

\section{affine Springer fibers}

\subsection{spectral curves}

Calculations in special cases show why the FL is essentially geometric
in nature, rather than purely analytic or combinatorial. We recall a
favorite old calculation of mine of the orbital integrals for $\so(5)$
and $\sp(4)$, the rank two odd orthogonal and symplectic Lie
algebras~\cite{hyperelliptic-curves}.  Let $F_v$ be a nonarchimedean
local field of residual characteristic greater than $2$.  Let $k$ be
the residue field with $q$ elements.  Choose $a\in \cc(F_v)$  and let $0,\pm t_1,\pm t_2$ be
the eigenvalues of the Kostant section $\gamma=\epsilon(a)$ in $\so(5)\subset \gl(5)$.
Assume that there is an odd
natural number $r$ such
\[
|\alpha(\gamma)| = q^{-r/2},
\]
for every root $\alpha$ of $\so(5)$.  We use the eigenvalues to
construct an elliptic curve $E_a$ over $k$, given by $y^2 =
(1-x^2\tau_1)(1-x^2\tau_2)$, where $\tau_i$ is the image of
$t_i^2/\varpi^r$ in the residue field, for a uniformizer $\varpi$.  By
direct calculation we find that the stable orbital integral
$\SO(a,f)$ of a test function $f$ equals the number of points on the
elliptic curve:
\begin{equation}\label{eqn:elliptic}
A(q) + B(q)\, \card( E_a(k)),
\end{equation}
up to some rational functions $A$ and $B$, depending on $f$.

Similarly, in the group $\sp(4)\subset\gl(4)$, there is an element
${{a'}}$ with related eigenvalues $\pm t_1,\pm t_2$.  According to the
general framework of (twisted) endoscopy, there should be a
corresponding function $f'$ on $\sp(4)$ such that the stable orbital
integral $\SO(a',f')$ in $\sp(4)$ is equal to (\ref{eqn:elliptic}).  A
calculation of the orbital integral of $f'$ gives a similar formula,
with a different elliptic curve $E'_{{a'}}$, but otherwise identical
to (\ref{eqn:elliptic}).  The elliptic curves $E_a$ and $E'_{{a'}}$
have different $j$-invariants (which vary with $a$ and ${{a'}}$).  The
proof of the desired identities of orbital integrals in this case is
obtained by producing an isogeny between $E_a$ and $E'_{{a'}}$.  (The
 identities of orbital
integrals are quite nontrivial, even though the Lie algebras $\so(5)$
and $\sp(4)$ are abstractly isomorphic.)

In a similar way in higher rank,  the spectral curves
\[
y^2 = (1-x^2 \tau_1)(1-x^2 \tau_2)\cdots (1-x^2 \tau_n)
\]
appear  in calculations of orbital integrals for
$\so(2n+1)$.  When orbital integrals are computed by brute force,
these curves appear as freaks of nature.  As it turns
out, they are not freaks at all, merely perverse.  One of the major
challenges of the proof of the FL and one of the major
triumphs of Ng\^o has been to find the natural geometrical setting that combines
orbital integrals and spectral curves.

\subsection{orbital integrals as affine Springer fibers}\label{sec:coset}

An orbital integral can be computed by solving a coset counting
problem.  The value of the integrand (\ref{eqn:k-orbital}) is
unchanged if $g$ is replaced with any element of the coset
$g\,G(\O_v)$.  The integral is thus expressed as a discrete sum over
cosets of $G(\O_v)$ in $G$ modulo the group action by $I_\gamma$.
Each coset $g\, G(\O_v)$ contributes a root of unity
$\langle\kappa,\gamma\rangle$ or $0$ to the value of the integral
depending on whether $\op{Ad}\,g^{-1}\gamma \in \g(\O_v)$ (again
modulo symmetries $I_\gamma$).  This interpretation as a coset
counting problem makes the FL appear to be a matter of simple
combinatorics.  However, purely combinatorial attempts to prove the FL
have failed (for good reason).

Let $\M_v(a,\bar k)$ be the set of cosets that fulfill the support
condition (\ref{eqn:k-orbital}) of the integral over $\bar k$:
\[
\M_v(a,\bar k) = \{g\in G(\bar F_v)/G(\bar \O_v) 
\mid \op{Ad}\,g^{-1} \gamma_0 \in \g(\bar\O_v)\}, 
\quad \gamma_0 = \epsilon(a).
\]

Kazhdan and Lusztig showed that the coset space $G(\bar
F_v)/G(\bar\O_v)$ is the set of $\bar k$-points of an inductive limit
of schemes called the {\it affine Grassmannian}.  Moreover, $\M_v(a,\bar k)$ itself
is the set of points of an ind-scheme $\M_v(a)$, called the {\it affine
Springer fiber}~\cite{KL:1988}.

Each irreducible component of $\M_v(a)$ has the same dimension.  This
dimension, $\delta_v(a)$, is given by a formula of
Bezrukavnikov~\cite{Bezrukavnikov}.  From that formula, it follows
that the dimension of the affine Springer fiber of $\nu(a)$ in $G$
exceeds the dimension of the affine Springer fiber of $a$ in $H$ by
precisely $r_v(a)$.  The factor $q^{r_v(a)}$ that appears in the FL is
forced to be what it is because of this simple dimensional analysis.

Goresky, Kottwitz, and MacPherson made an extensive investigation of
affine Springer fibers and conjectured that their 
cohomology groups are pure.  Assuming this conjecture, they prove the
FL for elements whose centralizer is an unramified Cartan
subgroup~\cite{GKM:2004}.  They prove the purity result in particular
cases by constructing pavings of the affine Springer
fibers~\cite{GKM:2006}.  

Laumon has made a systematic investigation of
the affine Springer fibers for unitary groups.  Ng\^o joined the effort, and together they
succeeded in giving a complete proof of the FL for unitary groups~\cite{LN:08}.

Ng\^o encountered two major obstacles in trying to generalize this
earlier work to an arbitrary reductive group.  These approaches
calculate the equivariant cohomology by passing to a fixed point set in $\M_v(a)$
under a torus action.  (In the case of unitary
groups, over a quadratic extension each endoscopic group becomes
isomorphic to a Levi subgroup of $GL(n)$.  The torus action comes from
the center of this Levi.)  However, in general, a nontrivial torus
action on the affine Springer fiber simply doesn't exist.

The second serious obstacle comes from the purity conjecture itself.
In accordance with Deligne's work, Ng\^o believed that the task of
proving purity results should become easier when the affine Springer
fibers are combined into families rather than treated in isolation.
With this in mind, he started to investigate families varying over a
base curve $X$.  This moves us from local geometry of a $p$-adic
field $F$ to the global geometry of the function field of $X$.  He
found that the Hitchin fibration is the global analogue of affine
Springer fibers.  The Hitchin fibers will be
described in the next section.  Deligne's purity theorem applies in this
setting~\cite{Deligne:Weil2}.

\section{Hitchin fibration}

The Hitchin fibration was introduced in 1987 in the context of
completely integrable systems~\cite{Hitchin:87}.  Roughly, the Hitchin
fibration is the stack obtained when the characteristic map
$\g\to\cc$ varies over a curve $X$.  Ng\^o carries out all geometry
 in the language of stacks without compromise, as developed in
\cite{LMB:2000}.  For this reason, {\it groupoids} (a category in which
every morphism is invertible) appear with increasing frequency throughout this report.

Fix a smooth projective curve $X$ of genus $g$ over a finite field
$k$.  We now shift perspective and notation, allowing the
constructions in Lie theory from previous sections to vary over the
base curve $X$.  In particular, we now let $G$ be a quasi-split
reductive group over $X$ that is locally trivial in the etale topology on
$X$.  Let $\g$ be its Lie algebra $G$ and $\cc$ the space of
characteristic polynomials, both now schemes over $X$.

Let $D$ be a line bundle on $X$. For technical reasons (stemming from
the $2$ in the structure constants of $\sl_2$), we assume that $D$ is
the square of another line bundle.  At one point in Ng\^o's proof of
the FL, it is necessary to allow the degree of $D$ to become
arbitrarily large~(\ref{sec:lmf}).  We place a subscript $D$ to
indicate the tensor product with $D$: $\g_D = \g\otimes_{\O_X}\!\!D$,
etc.

We let $\A$ be the space of global sections on $X$ with values in
$\cc_D=\cc\otimes_{\O_X}\!\!D$.  The group $G$ acts on $\g$ by the
adjoint action.  Twisting $\g$ by any $G$-torsor $E$ gives a
vector bundle $\op{Ad}(E)$ over $X$.

\begin{definition}
  The {\it Hitchin fibration} $\M$ is the stack given as follows.  For any
  $k$-scheme $S$, $\M(S)=[\g_D/G](X\times S)$ is the groupoid whose
  objects are pairs $(E,\phi)$, where $E$ is $G$-torsor over $X\times
  S$ and $\phi$ is a section of $\op{Ad}(E)_D$.
\end{definition}

There exists a morphism $f:\M\to\A$, obtained as a ``stacky''
enhancement of the characteristic map $\chi:\g\to\cc$ over $X$.  In
greater detail, $\chi:\g\to\cc$ gives successively
\[
%\begin{array}{lll}
 \hbox{}[\g_D/G]\to \cc_D,\qquad
 \hbox{}[\g_D/G](X\times S)  \to \cc_D(X\times S), \qquad
 \M(S) \to \A(S), \qquad
f:\M \to \A.
%\end{array}
\]
In words, the characteristic polynomial of $\phi$ is a section of
$X\times S$ with values in $\cc_D$; that is, an element of $\A(S)$.
We write $\M_a$ for the fiber of $\M$ over $a\in \A$.  This is the {\it Hitchin fiber}.

The centralizers $J_a$, as we vary $a\in \cc$, define a smooth group
scheme $J$ over $\cc$.  Now select on $\A$ an $S$-point: $a:S\to\A$.
There is a groupoid $\P_a(S)$ whose objects are $J_a$-torsors on
$X\times S$.  Moreover, $\P_a(S)$ acts on $\M_a(S)$ by twisting a pair
$(E,\phi)$ by a $J_a$-torsor.  As the $S$-point $a$ varies, we obtain
a {\it Picard stack} $\P$ acting fiberwise on the Hitchin fibration $\M$.  

\begin{example} We give an extended example with $G=GL(V)$, the
  general linear group of a vector space $V$.  In its simplest form, a
  pair $(E,\phi)$ is what we obtain when we allow an element $\gamma$ of
  the Lie algebra $\op{end}(V)$ to vary continuously along the curve
  $X$.  As we vary along the curve, the vector space $V$ sweeps out a
  vector bundle $E$ on $X$, and the element $\gamma\in\op{end}(V)$ sweeps
  out a section $\phi$ of the bundle $\op{end}(E)_D$.

  For each pair $(E,\phi)$, we evaluate the characteristic map
  $v\mapsto \chi(\phi_v)$ of the endomorphism $\phi$ at each point
  $v\in X$. This function belongs to the set $\A$ of a global sections
  of the bundle $\cc_D$ over $X$.  This is the morphism $f:{\M}\to
  {\A}$.

  {\it Abelian varieties} occur naturally in the Hitchin fibration.
  For each section $a=(c_{n-1},\ldots,c_0)\in \A$, the characteristic
  polynomial
\begin{equation}\label{eqn:spectral}
t^n + c_{n-1}(v) t^{n-1} + \cdots+ c_0(v)=0,\quad v\in X,
\end{equation} 
defines an $n$-fold cover $Y_a$ of $X$ (called the {\it spectral
  curve}).  By construction, each point of the spectral curve is a
root of the characteristic polynomial at some $v\in X$.  We consider
the simple setting when $Y_a$ is smooth and the discriminant of the
characteristic polynomial is sufficiently generic.  A pair $(E,\phi)$
over the section $a\in\A$ determines a line (a one-dimensional eigenspace
of $\phi$ with eigenvalue that root) at each point of the spectral
curve, and hence a line bundle on $Y_a$.  This establishes a map from
points of the Hitchin fiber over $a$ to $\op{Pic}(Y_a)$, the group of
line bundles on the spectral curve $Y_a$.  Conversely, just as linear
maps can be constructed from eigenvalues and eigenspaces, Hitchin
pairs can be constructed from line bundles on the spectral curve
$Y_a$.  The identity component $\op{Pic}^0(Y_a)$ is an abelian
variety.
\end{example}

\subsection{proof strategies}

At this point in the development, it would be most appropriate to
insert a book-length discussion of the geometry of the Hitchin
fibration, with a full development and many examples.  As Langlands
speculates in his review of Ng\^o's poof, ``an exposition genuinely
accessible not alone to someone of my generation, but to
mathematicians of all ages eager to contribute to the arithmetic
theory of automorphic representations, would be, perhaps, \ldots close
to 700 pages'' \cite{L:Ngo}.

To cut  $700$ pages short, what are the essential ideas?  

First, as mentioned above, the Hitchin fibration is the correct global
analogue of the (local) affine Springer fiber.  The relationship
between the Hitchin fiber $\M_a$ and the affine Springer fiber
$\M_v(a)$ can be precisely expressed as a factorization of
categories~(\ref{eqn:product-groupoid}): $\M_a$ modulo symmetries as a
product of $\M_v(a)$ modulo their symmetries as $v$ runs over closed
points of $X$.  Through the affine Springer fibers, the Hitchin
fibration can be used to study orbital integrals and the FL.

Second, the Hitchin fibration should be understood insofar as possible
through its Picard symmetries $\P$.  The obvious reason for this is
that it is generally a good idea to study symmetry groups.  The deeper
reason for this has to do with endoscopy.  The objects of the Picard
stack are torsors of the centralizer $J_a$.  Although the relationship
between $G$ and $H$ is mediated through dual groups, the relationship
between centralizers is direct: over $\cc_H$, there is a canonical
homomorphism from the regular centralizer $J$ of $G$ to the regular
centralizer $J_H$ of $H$:
\begin{equation}\label{eqn:JH}
\nu^*J\to J_H.
\end{equation}
Thus, their respective Picard stacks  are also
directly related and information passes fluently between them.  We should try to prove the FL
 largely at the level of Picard stacks.

 Third, by working directly with the Hitchin fibration, the difficult
 purity conjecture of Kottwitz, Goresky, and MacPherson can be
 bypassed.  Finally, continuity arguments may be used, as explained in~(\ref{sec:continuity}).

\subsection{perverse cohomology sheaves}

We give a brief summary without proofs of some of the main results
proved by Ng\^o about the perverse cohomology sheaves of the Hitchin
fibration.

There is an etale open subset $\tA$ of $\A\otimes_k\bar k$ that has
the technical advantage of killing unwanted monodromy.  The tilde will
be used consistently to mark quantities over $\tA$.  For example, if
we write $f^\a:\M^\a\to\A^\a$ for the Hitchin fibration, restricted to
the open set of anisotropic elements of $\A$, then $\tilde
f^\a:\tM^\a\to\tA^\a$ is the corresponding Hitchin fibration over the
anisotropic part of $\tA$.

 The conditions of Deligne's purity theorem are satisfied, so that $\tilde f_*^\a 
\bar{\ring{Q}}_\ell$ is isomorphic to a direct sum of perverse cohomology sheaves:
\[
{}^p\!H^n(\tilde f^\a_* \bar{\ring{Q}}_\ell)[-n].
\]

The action of $\tilde\P^\a$ on $\tilde\M^\a$ gives an action on the
the perverse cohomology sheaves, which factors
through the sheaf of components $\pi_0=\pi_0(\tilde\P^\a)$. 
The sheaf $\pi_0$ is an explicit quotient of the constant sheaf $X_*$ of
cocharacters, and hence $X_*$ acts on the perverse cohomology sheaves
through $\pi_0$.  As a result,  the perverse cohomology
sheaves break into a direct sum of $\kappa$-isotypic pieces
\begin{equation}\label{eqn:L}
L_\kappa = {}^p\!H^n(\tilde f^\a_* \bar{\ring{Q}}_\ell)_\kappa,
\end{equation}
as $\kappa$ runs over elements in the dual torus $\hat T$.  (By
duality, the cocharacter group $X_*$ is the group of characters of the
dual torus, which gives the pairing between $\hat T$ and $X_*$.)

We use the same curve $X$ and same line
bundle $D$  both for $G$ and for its endoscopic groups $H$.
The morphism  $\nu:\cc_H\to\cc$ from (\ref{eqn:nu}) 
extends to give $\nu:\cc_{H,D}\to \cc_D$ and then by taking sections of these
bundles, we obtain a morphism between their spaces of global sections:
\begin{equation}\label{eqn:nuA}
\nu:\A_H\to\A.
\end{equation}
We hope that no confusion arises by using the same symbol $\nu$ for all of these morphisms.

For each $\kappa\in\hat T$, there is a closed subspace of $\tA_\kappa$ of $\tA$
consisting of elements $a$ whose ``geometric monodromy'' lies in the centralizer
of $\kappa$ in the dual group ${}^LG$.  
The support of $L_\kappa$  lies in $\tA^\a_\kappa$.
Each subspace $\tA_\kappa$ is in fact
the disjoint union of the images of closed immersions 
\begin{equation}
\nu:\tA_H\to\tA
\end{equation}
coming from endoscopic groups $H$
with endoscopic data $(\kappa,\ldots)$. 
The geometric content of the FL is to be found in the comparison of $\nu^* L_\kappa$ with
\begin{equation}\label{eqn:LH}
L_{H,st} = {}^p\!H^{n+2r}(\tilde
  f^\a_{H,*}\bar{\ring{Q}}_\ell)_{st}(-r), \text{ where } r = \dim(\A) - \dim (\A_H).
\end{equation}
The subscript {\it st} indicates the isotypic piece with trivial character $\kappa=1$.

The anisotropic locus $\tA^\a$ admits a stratification by a numerical
invariant $\delta:\tA\to\ring{N}$:
\[
\tA^\a = \coprod_{\delta\in\ring{N}}\tA^\a_\delta.
\]
There is an open set $\tA^\good$ of $\tA^\a$, given as a union of some
strata $\tA^\a_\delta$ that satisfy:
\begin{equation}\label{eqn:codim}
\op{codim}(\tA_\delta^\a) \ge \delta.
\end{equation}

\subsection{support theorem}

The proof of the following  theorem about the support of the
perverse cohomology sheaves of the Hitchin fibration constitutes the
deepest part of the proof of the FL.

\begin{theorem}[support theorem]\label{lemma:support}
Let $Z$ be the support of a geometrically simple factor of $L_\kappa$.  
If $Z$ meets $\nu(\tA_H^\good)$ for some endoscopic group $H$  with 
 data $(\kappa,\ldots)$, 
then $Z=\nu(\tA_H^\a)$.  In fact, there is a unique such $H$.
\end{theorem}

A major chapter of the book-length proof of the FL is devoted to the
proof of the support theorem.  The strategy of the proof is to show
that every support $Z$ also appears as the support of some factor in
the ordinary cohomology of highest degree of the Hitchin fibration.
To move cohomology classes from one degree to another, Ng\^o uses
Poincar\'e duality and Pontryagin product operations on cohomology
coming from the action of the connected component of the identity
$\tP^{0,\a}$ on $\tM^\a$.  This action factors through the action of
an abelian variety, a quotient of the Picard stack $\tP^{0,\a}$.  To
show that the support $Z$ can be pushed all the way to the top degree
cohomology, it is enough to show that the dimension of this abelian variety is
sufficiently large and that the cohomology of the abelian variety acts
freely on the cohomology of the Hitchin fiber.  The required estimate on the
dimension of the abelian variety comes from the inequality (\ref{eqn:codim}).
Freeness relies on a polarization of the abelian variety.

Once the support $Z$ is known to appear as a support in the top
degree, he shows that the action of $\tP^\a$ on the Hitchin fibration
leads to an explicit description of the top degree ordinary cohomology
as the sheaf associated with the presheaf
\[
 U\mapsto \bar{\ring{Q}}_\ell^{\pi_0(\tP^\a)(U)}.
\]
The supports of $\pi_0$ can be described explicitly in terms
of data in the dual group, in the style of the duality theorems of
Kottwitz, Tate, and Nakayama.  By checking that the conclusion of the
support theorem holds for the particular sheaf $\pi_0$, the general support
theorem follows.
% Prop 6.5.1.

We apply the support theorem with $H$ as the primary reductive group
and $\kappa$ as the trivial character.  In this context, the only
endoscopic group of $H$ with stable data is $H$ itself.  Moreover,
$\nu$ is the identity map on $\tA_H$.  The support theorem takes the
following form in this case.

\begin{corollary}
Let $Z$ be the support of a geometrically
  simple factor of $L_{H,st}$.  If $Z$ meets $\tA_H^\good$, then
  $Z=\tA_H^\a$.
\end{corollary}

\subsection{continuity and the decomposition theorem}\label{sec:continuity}

The strategy that lies at the heart of the proof of the FL is a
continuity argument: arbitrarily complicated identities of orbital
integrals can be obtained as limits  of relatively
simple identities.

The complexity of an orbital integral is measured by the dimension of
its affine Springer fiber.  Growing linearly with $\deg_v(a^* \DIV)$,
this dimension is unbounded as a function of $a$.  Fortunately, globally, we can
view an element $a$ for which this degree at $v$ is large as a
limit of elements $a'$ with small degrees:
$\deg_w({a'}^*\DIV)\le 1$ for all $w\in X$.  This follows the
principle that a polynomial with repeated roots is a limit of
polynomials with simple roots.  When the degrees are at most $1$, the
affine Springer fibers have manageable complexity.
%, and as it grows, 
%we have unbounded unpleasantness.  

The Beilinson-Bernstein-Deligne-Gabber decomposition theorem for
perverse sheaves provides the infrastructure for the continuity
arguments~\cite{BBDG:1982}.  Let $S$ be a scheme of finite type over
$\bar k$.  The support $Z$ of a simple perverse sheaf on $S$ is a
closed irreducible subscheme of $S$.  There is a smooth open subscheme
$U$ of $Z$ and a local system ${\mathcal L}$ on $U$ such that the
simple perverse sheaf can be reconstructed as the middle extension of
the local system on $U$:
\[
i_* j_{!*} {\mathcal L}[\dim Z], \quad i:Z\to S,\quad j:U\to Z.
\]
We express this as a continuity principle: if two simple perverse
sheaves with the same support $Z$ are equal to the same local system
on a dense open $U$, then they are in fact equal on all of $S$.  

More generally, for any irreducible scheme $Z$ of finite type over
$k$, in order to show that two pure complexes on $Z$ are equal in the
Grothendieck group, it is enough to check two conditions:
\begin{enumerate}
\item Every geometrically simple perverse sheaf in
either complex has support all of $Z$.
\item Equality holds in the
Grothendieck group on some dense open subset $U$ of $Z$.
\end{enumerate}

The purpose of the support theorem~(\ref{lemma:support}) and
its corollary is to give the first condition for the two pure
complexes $\nu^* L_\kappa$ and $L_{H,st}$.  The idea is that second
condition should be a consequence of identities of orbital integrals
of manageable complexity, which can be proved by direct calculation.
The resulting identity of pure complexes on all of $Z$ should then
imply identities of orbital integrals of arbitrarily complexity.
This is Ng\^o's strategy to prove the FL.

\section{mass formulas}

\subsection{groupoid cardinality (or mass)}

Let ${\C}$ be a groupoid %(that is, a category in which every arrow is
%invertible) 
that has finitely many objects up to isomorphism and
in which every object has a finite automorphism group.  Define the
{\it mass} (or {\it groupoid cardinality}) of $\C$  to be the rational
number
\[
\mu(\C)= \sum_{x\in \op{obj}(\C)/\text{iso}} \frac{1}{\op{card}(\op{Aut}(x))}.
\]

\begin{example}
  Let ${\C}$ be the category whose objects are the
  elements of a given finite group $G$ and arrows are given by $x \mapsto g^{-1}
  x g$, for $g\in G$.  Then the set of objects up to isomorphism
  is in bijection with the set of conjugacy classes, the automorphism
  group of $x$ is the centralizer of $x$, and the mass is
\[
\mu(\C) = \sum_{x/\text{iso}} \frac{1}{\op{card}(\op{Aut}(x))} = 
\sum_{x/\text{iso}} \frac{\op{card}(\op{orbit}(x))}{\card{\,G}} = 1.
\]
\end{example}

\begin{example}  Let $P$ be a group that acts simply transitively on a set $M$.
Let ${\C}$ be the category whose set of objects is $M$, and let the set of morphisms
be given by the group action of $P$ on $M$.  There is one object up to isomorphism
and its automorpism group is trivial.  The mass of $\C$ is $1$.
\end{example}

\begin{example}\label{ex:groupoid}
  The following less trivial example appears in Ng\^o.  Let $P$ be the
  group $\ring{G}_m\times \ring{Z}$ defined over a finite field $k$ of
  cardinality $q$.  Let $M = (\ring{P}^1\times\ring{Z})/\sim$, where
  the equivalence relation $(\sim)$ identifies the point $(\infty,j)$ with
  $(0,j+1)$ for all $j$.  Thus, $M$ is an infinite string of
  projective lines, with the point at infinity of each line joined to
  the zero point of the next line.  The group $P$ acts on $M$ by
  $(p_0,i)\cdot (m_0,j) = (p_0 m_0,i+j)$, where $p_0m_0$ is given by
  the standard action of $\ring{G}_m$ on $\ring{P}^1$, fixing $0$ and
  $\infty$.  Let $\sigma$ be the Frobenius automorphism of $\bar k/k$,
  and define a twisted automorphism of $P(\bar k)$ and $M(\bar k)$ by
  $\sigma(x_0,i) = (\sigma x_0^{-1},-i)$.  Define a category $\mathcal
  C$ with objects given by pairs
\begin{equation}\label{eqn:objects}
(m,p)\in M(\bar k)\times P(\bar k) \text{ such that } \sigma(m) = p
m.
\end{equation}
Define  arrows by $h\in P(\bar k)$, where
\begin{equation}\label{eqn:arrows}
h(m,p) = (m',p'),    \text{ provided } hm = m' \text{ and } h p = p'\sigma(h).
\end{equation}  
Then it can be checked by a direct
calculation that there are two isomorphism classes of objects in this
category, represented by the objects
\[
((0,1),(1,1))\text{ and }  ((1,0),(1,0))\in M(\bar k)\times P(\bar k) = 
(\ring{P}^1(\bar k)\times\ring{Z}) \times (\ring{G}_m(\bar k)\times\ring{Z}).
\]
The group $P(\bar k)^\sigma$ of order $q+1$ acts as automorphisms of the first object,
and the group of automorphisms of the second object is trivial.  The mass of this category
is therefore
\[
\mu(\C) = \frac{1}{q+1} + 1.
\]
\end{example}

More generally, suppose there exists a function $\op{Obj}(\C)\to  A$ from the objects
of a groupoid into in a finite abelian group $A$ and that the image in $A$ of each object
depends only its isomorphism class.  Then for every
character $\kappa$ of $A$, we can define a {\it $\kappa$-mass}:
\[
\mu_\kappa(\C)= \sum_{x\in \op{obj}(\C)/\text{iso}} 
\frac{\langle\kappa,x\rangle}{\op{card}(\op{Aut}(x))}.
\]

\begin{example}
  In  example~\ref{ex:groupoid}, if $(m,p)$ is an object and
  $p=(p_0,j)\in \ring{G}_m\times\ring{Z}$, then the image of $j$ in
  $A=\ring{Z}/2\ring{Z}$ depends only on the isomorphism class of the
  object $(m,p)$.  If $\kappa$ is the nontrivial character of $A$,
  then the $\kappa$-mass of this groupoid is
\[
\mu_\kappa(\C) = -\frac{1}{q+1} + 1.
\]
\end{example}

\subsection{mass formula for orbital integrals}

Let $\M_v(a)$ be the affine Springer fiber for the element $a$ and let
$J_a$ be its centralizer.  We write $\P_v(J_a)$ for the group of
symmetries of the affine Springer fiber.   Let $\C$ be the groupoid of $k$-points of
the quotient stack $[\M_v(a)/\P_v(J_a)]$ with objects $(m,p)$ and morphisms
and $h$ defined by the earlier formulas (\ref{eqn:objects}) and (\ref{eqn:arrows}),
(substituting $\M_v(a)$ for the space $M$ and $\P_v(J_a)$ for the symmetries $P$).

For each character of $H^1(k,\P_v(J_a))$ we can naturally define a
character $\kappa$ of $H^1(F_v,J_a)$ as well as a character (also
called $\kappa$) on a finite abelian group $A$ as above.

The description of orbital integrals in terms of affine Springer fibers
 takes the following form.  It is a variant of the coset arguments of~(\ref{sec:coset}).

\begin{theorem}\label{lemma:orbital-mass}
For each regular semisimple element $a\in \cc(\O_v)$, 
the $\kappa$-mass of the category $\C$ is equal to the
$\kappa$-orbital integral of $a$:
\[
\mu_\kappa(\C) = c\,\, \OO_\kappa(a),
\]
up to a constant $c=\op{vol}(J^0_a(\O_v),dt_v)$ used to normalize measures.
\end{theorem}
% Prop 8.2.7.

\subsection{product formula for masses}

Recall from (\ref{eqn:nuA}) that there is a morphism $\nu:\A_H\to\A$.
We choose a commutative group scheme $J'_a$
for which there are homomorphisms
\begin{equation}\label{eqn:JH'}
J'_a \to J_{\nu(a)} \to J_{H,a}.
\end{equation}
extending the homomorphism (\ref{eqn:JH}) and that become isomorphisms
over a nonempty open set $U$ of $X$.  The group scheme $J'_a$ can be
chosen to satisfy other simplifying assumptions that we will not list
here.  The homomorphisms (\ref{eqn:JH'}) functorially determine an
action of $\P(J'_a)$ on both Hitchin fibrations $\M_{\nu(a)}$ and
$\M_{H,a}$.  Changing notation slightly, we will assume that
henceforth all masses for both $G$ and $H$ are computed with respect
to the same Picard stack $\P(J'_a)$ in global calculations and with respect to
$\P_v(J'_a)$ in local calculations.  This simplifies the comparisons
of masses that follow.

For each element $a\in \A_H^\a(k)$, we have a  mass $\mu_H(a)$ of
the groupoid of $k$-points of the Hitchin fiber $\M_{H,a}$ modulo
symmetry on $H$.  Its image $\nu(a)\in \A^\a$ has a
$\kappa$-mass of the groupoid of $k$-points of the Hitchin fiber
$\M_a$ modulo symmetry.

For each regular semisimple element $a\in \cc_H(\O)$, we have a mass
of the affine Springer fiber modulo symmetry on $H$. We write
$\mu_{H,v}(a)$ for this mass.  Moreover, if the image $\nu(a)$ under
the map $\nu:\cc_H\to\cc$ is also regular semisimple, there a
$\kappa$-mass $\mu_{\kappa,v}(\nu(a))$ of the affine Springer fiber
modulo symmetry of $\nu(a)$ in $G$.

$\A_H$ is the set of global sections of $\cc_{H,D}$ over $X$.  For
each $v\in X$, we can fix a local trivialization of $\cc_{H,D}$ at $v$
and evaluate a section $a\in \A_H$ at $v$ to get an element
$a_v\in\cc_H$.  We write $\M_{H,v}(a) = \M_{H,v}(a_v)$ for its affine
Springer fiber, and $\mu_{H,v}(a)$ for the local mass $\mu_{H,v}(a_v)$.
Similarly, we write $\mu_{\kappa,v}(\nu(a))$ for
$\mu_{\kappa,v}(\nu(a_v))$.  With all of these conventions in place, we can state
the product formula:

\begin{theorem}\label{lemma:product}
  Let $a\in \A_H^\a(k)$.  The mass of a Hitchin fiber modulo symmetry 
   satisfies a product formula over all
  closed points of $X$ in terms of the masses of the individual affine
  Springer fibers modulo symmetries:
\[
\mu_\kappa(\nu(a)) =\prod_{v\in X} \mu_{\kappa,v}(\nu(a)), 
\quad \mu_H(a) = \prod_{v\in X} \mu_{H,v}(a).
\]
The local factors are $1$ for almost all $v$ so that the products are in fact finite.
\end{theorem}

This theorem is a geometric version of the factorization of
$\kappa$-orbital integrals over the adele group into a product of
local $\kappa$-orbital integrals in \cite{Langlands:debuts}.  It
confirms the claim that the Hitchin fibration is the correct global
analogue of the affine Springer fiber.

\begin{proof}[proof sketch]
The proof choses an open set of $X$ over which $J'_a$ is isomorphic to
$J_a$. For a given $a$, on a possibly smaller open set $U$ of $X$, the
action of $\P(J_a)$ on $\M_a$ induces an isomorphism of $\P(J_a)$ with
$\M_a$.  It follows that the local masses equal $1$ for all $v\in U$.
The product in the lemma can be taken as extending over the
finite set of points $X\setminus U$.  The lemma is a consequence of a wonderful product
formula for stacks, relating the Hitchin fibration to affine Springer fibers:
\begin{equation}\label{eqn:product-groupoid}
[\M_a/\P(J'_a)] = \prod_{v\in X\setminus U} [\M_{v}(\nu(a))/\P_{v}(J'_a)].
\end{equation}
A similar formula holds on $H$.
\end{proof}

\subsection{global mass formula}

%It is known that the codimension $r$ of $\nu^*\A_H$ in $\A_G$ is given
%by the formula
%\[
%r = \dim \A - \dim \A_H =  (\dim\,G - \dim\,H) \deg(D)/2.
%\]

The following is the key global ingredient of the proof of the FL.  In
fact, it can be viewed as a precise global analogue of the FL.

\begin{theorem}[global mass formula]\label{lemma:gmf}
  Assume $\deg(D)>2g$, where $g$ is the genus of $X$.  Then for all
  $\tilde a\in \tA_H^\good(k)$ with images $a\in\A_H^\a(k)$ and
  $\nu(a)$ in $\A(k)$, the following mass formula holds:
\[
\mu_\kappa(\nu(a)) = q^{r} \mu_H(a),\text{ where } r = \dim\A - \dim \A_H.
\]
\end{theorem}

\begin{proof}[proof sketch]
  The proof first defines a particularly nice open set $\tU$ of
  $\tA_H^{\good}\subset\tA_H^\a$.  The idea is to place conditions on $\tU$ to make
  it as nice as possible, without imposing so many conditions that it
  fails to be open.  There exists an open set $\tU$ of $\tA_H^\good$
  on which both of the following conditions hold:
\begin{itemize}
\item Each $\tilde a\in \tU(\bar k)$ cuts the divisor $\DIV_{H,D}+\RDIV_{D}$ transversally.
\item For each $n$, the restriction to $\tU$ of the perverse
   cohomology sheaves $\nu^* L_\kappa$ and $L_{H,st}$ from (\ref{eqn:L}) and (\ref{eqn:LH})
% $\nu^*\, {}^p\!H^n(\tilde f^\a_*\bar
%  {\ring{Q}}_\ell)_\kappa$ and ${}^p\!H^{n+2r}(\tilde
%  f^\a_{H,*}\bar{\ring{Q}}_\ell)_{st}(-r)$ 
are pure local systems of
  weight $n$.
\end{itemize}
The support theorem (\ref{lemma:support}) and decomposition and
continuity strategies~(\ref{sec:continuity}) are used to find the pure
local systems.

After choosing $\tU$, the proof of the lemma establishes the global
mass formula on $\tU$, then extends it to all of $\tA_H^{\good}$. 

By imposing such nice conditions on $\tU$, Ng\^o is able to prove the
mass formula on this subset by explicit local calculations.  By the
transversality condition on $\tilde a$, at any given point $v$, the local
degree $(d_{H,v}(\nu(a)),r_{v}(a))$ must be $(0,0)$, $(1,0)$, or $(0,1)$.
From Bezrukavnikov's dimension formula~(\ref{sec:coset}), the dimension of the endoscopic affine
Springer fiber $\M_{H,v}(a)$ is $0$.  In fact, $\P_v(J'_{a})$ acts
simply transitively on the affine Springer fiber, and the mass is $1$.

It is therefore enough to compute  the $\kappa$-mass of $\nu(a)$ and compare.
The transversality condition  determines the
possibilities for the dimension  $\delta_v(\nu(a))$ in $G$.
The affine Springer fiber in this case is at most one and the
$\kappa$-masses of the groupoids can be
computed directly.  In fact, (\ref{ex:groupoid}) is a typical example of the
computations involved.  

The result of these local calculations is that for every point $\tilde
a$ in $\tU$, with images $\a\in\A_H$ and $\nu(a)\in \A$, a local mass
formula holds for all closed points $v$ of $X$:
\begin{equation}\label{eqn:local-mass}
\mu_{\kappa ,v}(\nu(a)) = q^{\deg(v) r_v(a)} \mu_{H,v} (a).
\end{equation}
The exponents satisfy
\begin{equation}\label{eqn:exp}
r = \sum_v \deg(v) r_v(a).
\end{equation}
These two identities, together with the product formula for the global mass, give
the lemma for elements $a$ of $\tU$.  

The extension from $\tU$ to all of $\tA_H^{\good}$ is a global
argument.  Through the Grothendieck-Lefschetz trace formula (adapted
to stacks), this identity of global masses over $\tU$ can be expressed
as an identity of alternating sums of trace of Frobenius on local
systems.  These calculations can be repeated for all finite extensions
$k'/k$.  By Chebotarev density as we vary $k'$, the
semisimplifications of the local systems are isomorphic on $G$ and
$H$.

Following the decomposition and continuity strategy~(\ref{sec:continuity}), 
this isomorphism of local systems on $\tU$ extends to an 
isomorphism between (the semisimplifications of)
$\nu^* L_\kappa$ and $L_{H,st}$.  This isomorphism, again by
Grothendieck-Lefschetz, translates back into a mass formula for  the
Hitchin fibration modulo symmetries, and hence the result.
\end{proof}

\subsection{local mass formula and the FL}\label{sec:lmf}

We recall some notation from Section~\ref{sec:statement}.
Let $G_v$ be a reductive group scheme over the ring of integers $\O_v$ of
a nonarchimedean local field $F_v$ in positive characteristic.  Let $q$ be the cardinality
of the residue field $k$.    Let $(\kappa,\rho)$ be
endoscopic data defining an endoscopic group $H_v$.  Let $a \in
\cc_H(\O_v)$ and $\nu(a)$ be its image in $\cc(F)$.  Assume that $\nu(a)$
is regular semisimple.  Let $r_v(a)\in\ring{N}$ be the local invariant.

Assume that the characteristic of $k$ is large.  By
standard descent arguments (\ref{sec:descent}), we also assume
without loss of generality that the center of $H_v$ does not contain a
split torus.

\begin{theorem}[local mass formula]\label{lemma:lmf}
  The following local mass formula holds for general anisotropic
  affine Springer fibers (both masses being computed with respect to
  the same symmetry group $\P_v(J'_a)$ acting on the fibers):
\[
\mu_{\kappa,v}(\nu(a)) = q^{\deg(v) r_v(a)}\mu_{H,v}(a).
\]
\end{theorem}

\begin{corollary}[fundamental lemma (FL)]\label{lemma:fl}
$$\OO_\kappa(\nu(a)) = q^{r_v(a)}\SO_H(a).$$
\end{corollary}

The corollary follows from the theorem by the mass formula~(\ref{lemma:orbital-mass}) for orbital
integrals.

\begin{proof}[proof sketch]
  The proof of the FL is a global argument based on the global mass
  formula on $\tA_H^{\good}$.  We can use
  standard strategies to embed the local setting into a global
  context. We pick a smooth projective curve $X$ over $k$, such that a completion
  of the function field at some place $v$ is isomorphic to $F_v$ and $\deg(v)=1$.  We
  choose a global endoscopic groups $H$ of a reductive group $G$, a
  divisor $D$ on $X$, a global element\footnote{More accurately, Ng\^o
    shows that a suitable element $a'$ exists over every sufficiently
    large finite field extension $k'/k$.  He makes the global
    arguments over the extensions $k'$ and uses a Frobenius eigenvalue
    argument at the end to go back to $k$.}  $a'$ in the Hitchin base
  $\A_H$ of $H$.   These global choices are to specialize to the given data $G_v$ $H_v$ at $v$.
   If the degree of $D$ is sufficiently large, then we can assume that
  $a'$ is the image of $\tilde a' \in \A_H^{\good}(k)$.  The element $a'$ and its image
  $\nu(a')$ in $\A$ are chosen to approximate the given local elements
  $a$ and $\nu(a)$ so closely that  their affine Springer fibers
  together with their respective  symmetries are unaffected at $v$.

The global mass formula~(\ref{lemma:gmf}) for $\tilde a'$ asserts:
\[
\mu_\kappa(\nu(a')) = q^{r} \mu_H(a').
\]
 By the product formula~(\ref{lemma:product}) and (\ref{eqn:exp}),
each global mass is a product of local masses:
\begin{equation}\label{eqn:ph}
\prod_w \mu_{\kappa,w}(\nu(a')) = \mu_\kappa(\nu(a')) = 
q^{r}\mu_H(a') = \prod_{w} q^{\deg(w) r_w(a')}\mu_{H,w}(a').
\end{equation}

The global data is chosen in such a way that at every closed point
$w\ne v$, the transversality conditions hold, so that the calculation
(\ref{eqn:local-mass}) of the previous section gives the local mass
formula at $w$:
\[
\mu_{\kappa,w}(\nu(a')) = q^{\deg(w) r_w(a')}\mu_{H,w}(a'),\quad w\ne v.
\]  
These masses are nonzero and  can be canceled from 
the products in  (\ref{eqn:ph}).  What remains is a single uncanceled term on each side:
\[
\mu_{\kappa,v}(\nu(a')) = q^{\deg(v) r_v(a')}\mu_{v}(a'),\text{ with } \deg(v)=1.
\]
Since, $a'$ was chosen as a  close approximation of $a$ at  $v$, we also have
\[
\mu_{\kappa,v}(\nu(a)) = q^{r_v(a)}\mu_{v}(a).
\]
This is the desired local mass formula at  $v$.
\end{proof}

\section{uses of the FL}  \label{sec:uses}
% 6 pages.

``Lemma'' is a misleading name for the Fundamental Lemma because it
went decades without a proof, and its depth goes far beyond what would
ordinarily be called a lemma.  Yet the name FL is apt both because it
is fundamental and because it is expected to be used widely as an
intermediate result in many proofs.  This section mentions some major
theorems that have been proved recently that contain the FL as an
intermediate result.  In each case, the FL appears to be an
unavoidable ingredient.

The FL appears as a specific collection of identities
that are needed to stabilize the Arthur-Selberg trace formula.  
If $G$ is a reductive group defined over a number field $F$, 
the trace formula for $G$ is an identity of the general form
\[
\sum_{\gamma\in G(F)/\sim} c_\gamma \OO(\gamma,f) +\cdots = 
\sum_\pi m(\pi) \,\op{trace}\, \pi(f) + \cdots
\]
for compactly supported smooth functions $f$ on the adele group
$G(\ring{A_F})$.  On the left-hand side appears a sum of orbital integrals
and on the right-hand side the sum runs over discrete automorphic
representations $\pi$ of $G$.    The trace formula contains more complicated terms that have
been suppressed.

By {\it stabilization of the trace formula}, we mean that the terms on
the left-hand side of the trace formula that are associated with a
given stable conjugacy class have been gathered together, rearranged
into $\kappa$-orbital integrals, and then replaced with stable orbital
integrals on the endoscopic groups.  These manipulations are justified
by the FL and by a product formula that relates the adelic orbital
integrals $\OO(\gamma,f)$ to orbital integrals on local fields.
Another Bourbaki seminar gives further details about the role of the
the FL in the stable trace formula~\cite{Dat:2004}.  All applications
of the FL come through the stable trace formula.

Before going into recent uses of the FL, we might also mention various
special cases of the FL that have been known for years.  These classical
cases of the FL already give abundant evidence of the usefulness of
the lemma.  For example, Langlands proves the FL for
cyclic base change for $GL(2)$ in his book \cite[Lemma~5.10]{LBC:1980}.
From there, it enters into the proof of the tetrahedral and
octahedral cases of the Artin conjecture (the Langlands-Tunnell theorem),
which in turn is used by Wiles in the proof of Fermat's Last Theorem.
Waldspurger's proof of the FL for $SL(n)$ is taken up by Henniart and Herb in their
proof of local automorphic induction for $GL(n)$, which becomes part of the
proof of the proof of the local Langlands correspondence for $GL(n)$
in Harris and Taylor~\cite{Wald:1991},\cite{Herb:Autoinduct},\cite{Harris:Taylor:local}.

Shimura varieties provided much of the early motivation
for endoscopy and the FL~\cite{LZ:1979}.  When expressing the 
Hasse-Weil zeta function of Shimura varieties as a product of
automorphic $L$-functions, the formula involves the $L$-functions
associated with endoscopic groups $H$ as well as those of $G$.  This
can be most clearly through a comparison of the stable trace formula 
with the Grothendieck-Lefschetz trace formula of the Hasse-Weil zeta
function.
An early application of the FL carries this out for
Picard modular
varieties.~\cite{LPicard:1992}.  From there, the FL becomes relevant to the theory of Galois representations
through the representations associated with Shimura varieties.
% On the Zeta-Functions of Some Simple Shimura Varieties.
% http://publications.ias.edu/sites/default/files/simpshim-ps.pdf
%  Picard modular varieties.

We turn to more recent uses of the FL.  
For most applications to
date, the FL for unitary groups is used as well as the twisted FL
between $GL(n)$ and unitary groups.  Applications of the trace formula
to Shimura varieties often rely on a base change FL, which  arises
because of the description of that Kottwitz gives of points on certain
Shimura varieties in terms of twisted orbital
integrals~\cite{Kott:1990}.

% Kottwitz R.: Shimura varieties and $\lambda$-adic representations, in Automor- 

The original proof by Clozel, Harris, Shepherd-Barron and Taylor
of the Sato-Tate conjecture for elliptic curves over
$\ring{Q}$ was restricted to elliptic curves with non-integral
$j$-invariants~\cite{Car:Bourbaki}.
With the advent of the general FL, it has become possible to remove
the non-integrality assumption and to greatly extend the theorem, in particular 
to elliptic curves over a totally real number field~\cite{BGHT:2010}.

% Carayol's Bourbaki:
% http://www.mathematik.hu-berlin.de/gradkoll/Carayol_Exp.977.H.C4.pdf
% Harris's survey
%http://www.math.unipd.it/~algant/IHP_SMF.pdf

Shin and Morel use the FL in recent work on the cohomology of Shimura
varieties and associated Galois
representations~\cite{Shin:2010}~\cite{Morel:2010}.  Other advances
rely on their work.  In particular,
Skinner and Urban have proved the Iwasawa-Greenberg main conjecture
for many modular forms and in particular for the newforms associated
with many elliptic curves over
$\ring{Q}$~\cite{Skinner-Urban:2010},\cite{Skinner:2010}.  Their work
ultimately relies on the work of Shin and Morel and on the FL to prove
the existence of certain Galois representations.

Last year, Bhargava and Shankar proved that when elliptic curves $E$
over $\ring{Q}$ are ordered by height, a positive fraction of them
satisfy the Birch and Swinnerton-Dyer conjecture~\cite{BS:2010}.
Specifically, a positive fraction of them have rank $0$ and analytic
rank $0$.  First they construct a set (of positive density) of
elliptic curves with rank $0$.  Second, they construct a subset (again
of positive density) of the rank $0$ set, consisting of elliptic
curves with analytic rank $0$.  This second step relies on conditions
in Skinner and Urban for the analytic rank to be zero, and hence
indirectly on the FL.

Moeglin classifies the discrete series representations of unitary
groups over a nonarchimedean local field~\cite{Moeglin:2007}.  Again,
this relies on the FL for unitary groups and related
variants. Finally, we mention that Arthur's forthcoming book uses the
twisted FL between $GL(n)$ and the classical
groups~\cite{Arthur:2011}.  His work uses the trace formula to give a
classification of the discrete automorphic representations of
classical groups in terms of cuspidal automorphic representations of
$GL(n)$.  It also gives a classification locally, for $p$-adic fields.

I will leave a further discussion of the uses of the FL to those whose research
in this area is fresher than my own.

\section{reductions}\label{sec:reduce}

Langlands first expressed the FL in these words:
``Mais m\^eme apr\`es avoir
v\'erifi\'e que les facteurs de
transfert existent, il reste \`a v\'erifier ce que j'appelle le
lemme fondamental, qui affirme que pour des $G$, $H$ et $\phi_H$
non-ramifi\'es, on a $f\mapsto c\, \phi_H^*(f)$ $\ldots$ 
pour toute fonction $f\in {\mathcal H}_G$.''
\cite[p.49]{Langlands:debuts}

 In this notation, $\phi_H^*$ is the homomorphism given by the
 Satake transform, from the spherical Hecke algebra ${\mathcal H}_G$
  with respect to a hyperspecial maximal compact subgroup
of an unramified reductive group $G$ to the spherical Hecke algebra
 on $H$ .  The
 arrow $f\mapsto c\,\phi_H^*(f)$ is his assertion that for every
 strongly $G$-regular element $\gamma$ in $H$, the transfer (specified by
 transfer factors) of each $\kappa$-orbital integral of a spherical
 function $f$ on $G$ (over a stable conjugacy class in $G$ matching
 $\gamma$) is equal to the stable orbital integral of $\phi_H^*(f)$ on
 the stable conjugacy class of $\gamma$ in $H$.

 This final section describes some theorems related to the FL that
 simplify it from the form in which it was initially conjectured, to
 the final form in which it was proved by Ng\^o.  Waldspurger's work
 has been particularly significant in transforming the conjecture into
 a friendlier form.  In the initial conjecture, the existence
 of transfer factors was part of the conjecture.  Langlands and
 Shelstad later defined the transfer factors explicitly
 \cite{LS:1987}.  We also mention some extensions of the FL.

\subsection{descent to the Lie algebra}\label{sec:descent}

A lemma of Harish-Chandra's asserts the transfer of an orbital
integral on $G$ near a singular semisimple element $z\gamma_0$, with
$z$ central, to an orbital integral on the centralizer $I_{\gamma_0}$.
This is called the {\it descent} of orbital integrals.  Langlands and
Shelstad made hard calculations in Galois cohomology to prove that
their transfer factors are compatible with Harish-Chandra's descent of
orbital integrals~\cite{LS:1990}.  The point of their calculations was
to reduce identities of orbital integrals involving transfer factors
to a neighborhood of $\gamma_0=1$, arguing by induction on the
dimension of the centralizer.  In a neighborhood of $\gamma_0=1$, identities can
be pushed to the Lie algebra, using the exponential map.

The original FL has been supplemented by a twisted FL, conjectured by
Kottwitz and Shelstad, where the data is twisted by a nontrivial outer
automorphism $\theta$ of the group $G$~\cite{KS:1999}.  In the
untwisted case, the centralizer of an element fails to give a group of
smaller dimension precisely when the element is central.  By contrast,
a twisted centralizer (with respect to a nontrivial outer
automorphism) always has dimension less than $G$.  As a consequence,
descent {\it always untwists} the twisted FL into identities of
ordinary orbital integrals.  If the (standard) FL is then applied,
each $\kappa$-orbital integral can be replaced with a stable orbital
integral.  By combining both descent and stabilization, the twisted FL
of Kottwitz and Shelstad takes the form of identities of stable
orbital integrals on the Lie algebra (from which the automorphism and
the character $\kappa$ have entirely vanished).  The corresponding
long calculations in Galois cohomology that establish descent
properties of the transfer factors for the twisted FL have been
carried out by Waldspurger \cite{Wald:2008}.  Ng\^o proves the general
twisted FL in its untwisted stable form on the Lie algebra.

\subsection{Hecke algebras}

 A global argument based on
the trace formula shows that the FL holds for the full
Hecke algebra for an arbitrary nonarchimedean local field of
characteristic zero, provided it holds for the unit element of the
Hecke algebra for local fields of sufficiently large residual
characteristic (and for groups of smaller dimension)~\cite{FLSE}.  
The idea of the proof is to choose suitable global
functions for which the comparison of stable trace formulas yields an
obstruction to the FL.  This obstruction, which comes from the
spectral side of the trace formula, takes the form
of a set of linear functionals  
\[
L:{\mathcal H}\to\ring{C},\quad L(f) = \sum_\pi a(\pi) \op{trace}\,\pi(f),
\]
on the local spherical Hecke
algebra ${\mathcal H}$ of the reductive group $G$, each given by a finite sum over
irreducible admissible representations with an Iwahori fixed vector.
By purely local arguments, it
can be shown that no nonzero linear map $L$ exists of the form
prescribed by the global theory.  
Because the obstructions $L$ are zero,
the FL can be shown to hold on the full spherical Hecke algebra.

\subsection{smooth transfer}

Langlands's book on the stabilization of the trace formula contains
two separate conjectures:  the transfer of smooth
functions  and the FL~\cite{Langlands:debuts}.  An important result
of Waldspurger links the two conjectures, by proving that the
FL implies the transfer of smooth functions.  
His key local lemma shows how to obtain simultaneous control over the
orbital integrals of test functions $f$ on the Lie algebra $\g$ and the orbital integrals of
their Fourier transform $\hat f$~\cite[Prop.~8.2]{Wald:transfert}.  
In view of the uncertainty principle, it is a remarkable feat to
control both $f$ and $\hat f$ as he does.
His proof is a global argument, based on a stable Poisson summation
trace formula on the Lie algebra over the ring of adeles.  The key local lemma allows
Waldspurger to pick global test functions for which the comparison of
trace formulas asserts a local identity: the Fourier transform of a
semisimple $\kappa$-orbit on $G$ equals the Fourier transform of the
corresponding stable orbit on $H$.  By a purely local argument, this
stabilization identity of Fourier transforms  implies smooth transfer.

\subsection{weighted orbital integrals}

Langlands's book is a {\it d\'ebut}: he
stabilizes the terms in the trace formula that come from regular
elliptic conjugacy classes, but  this is insufficient for general
applications of the trace formula.  Kottwitz extended the analysis to
singular elliptic conjugacy classes~\cite{Kott:singular}.  Arthur has
completed the full stabilization without restrictions on the conjugacy
classes.  The non-elliptic conjugacy classes lead to significant
complications.  Arthur truncates the trace formula to obtain the
convergence of the non-elliptic terms.  Because of truncation, the
non-elliptic terms bear ``weights,'' non-invariant factors that
appears in the integrand of orbital integrals.  Arthur conjectured a
weighted FL needed for stabilization of the non-elliptic
terms~\cite{Arthur:2002}.  Chaudouard and Laumon have used the Hitchin
fibration to prove Arthur's weighted
FL~\cite{CL:2009:I},\cite{CL:2009:II}.

\subsection{transfer to characteristic zero}

The FL for nonarchimedean local fields in characteristic zero can be
deduced from the FL in positive
characteristic~\cite{Wald:2006},~\cite{CHL:2010}.  Cluckers and Loeser
have developed a general abstract theory of integration as a
combination of primitive operations such as taking the volume of a
ball of given radius, enumerating points on a variety over the residue
field, summing an infinite $q$-series, and making a change of
variables.  Since each of the primitive operations manifestly depends
only on the residue field rather than the field itself, their theory
allows many identities of integrals to be transfered from one field to
another with the same residue field.  The FL lemma and its weighted
and twisted variants are identities that fall within the scope of this
theory.  Waldspurger's approach is also an abstraction
of $p$-adic integration, but requires more detailed
properties of the specific integrals appearing in the FL.

\subsection{etc.}

These separate variations on the FL can be considered in concert: a
weighted twisted FL, the twisted FL on the full Hecke algebra,
transfer of the weighted FL to characteristic zero, and so forth.
Most combinations have now been proved.

\section{literature}

I recommend Ben-Zvi's video lecture for a presentation of the big
ideas of Ng\^o's proof.  Drinfeld's lecture notes contain many worked
examples and exercises that are helpful in learning the geometric
concepts.  I also recommend Nadler's survey~\cite{Nadler:2010},
Casselman for an in-depth treatment of $SL_2$ with history going back
to Hecke~\cite{Cass:2010}, my summer school lecture for the detailed
statement of the FL~\cite{Hales:FL-statement}, Arthur's Fields medal
laudation~\cite{Arthur:2010}, and \cite{CHLaumon:2010}.  Several
articles in the book project deal with the
FL~\cite{Harris:book-project}, particularly~\cite{DN:2010}.

Ng\^o's book is superb, both as mathematics and as
exposition~\cite{NBC:2010}.  It is helpful to read it with his earlier
paper~\cite{NBC:2006}.  He has several supplementary accounts,
especially the expository account~\cite{NBC:report:2010}, his article in the book
project~\cite{NBC:abelian}, and ICM lectures.

While there have been numerous applications that quote the FL as a
finished product, Yun, Chaudouard, and Laumon are noteworthy in
following Ng\^o in their direct use of the Hitchin fibration to prove
new results in the field~\cite{Yun:2009}.

I wish to thank Bhargava, Harris, Ng\^o, and Skinner for comments that
helped me to prepare this and my earlier report~\cite{thales:NBC:2011}.

\raggedright
\bibliographystyle{plain} %plainnat
%\bibliography{all}

\end{document}